\documentclass{scrartcl}
\usepackage[british]{babel}
\usepackage[textwidth=150mm,top=2cm]{geometry}

\usepackage{dsfont, amsthm, mathtools,  enumerate,amsmath,amsrefs,array}

\usepackage{dsfont}
\usepackage{amssymb}
\usepackage{aurical}
\usepackage{mathtools}
\usepackage{longtable}
\usepackage{wasysym}
\usepackage{euscript}
\usepackage{mathrsfs}
\usepackage{tikz}
\usepackage[hidelinks]{hyperref}

\usepackage{xcolor}

\usepackage{libertine}

\usepackage{todonotes}

\title{Decay estimate for the solution of the evolutionary damped $p$-Laplace equation }

\author{Farid \textsc{Bozorgnia}\footnote{Corresponding Author: bozorg@math.ist.utl.pt, CAMGSD, Instituto Superior T\'{e}cnico, University of Lisbon, Av. Rovisco Pais, 1049-001 Lisbon, Portugal.} \and Peter \textsc{Lewintan}\footnote{peter.lewintan@uni-due.de, University of Duisburg-Essen, Germany}}
\date{\today}

\theoremstyle{plain}
\newtheorem{theorem}{Theorem}[section]
\newtheorem{corollary}[theorem]{Corollary}
\newtheorem{lemma}[theorem]{Lemma}

\theoremstyle{definition}
\newtheorem{definition}[theorem]{Definition}

\theoremstyle{remark}
\newtheorem{remark} [theorem]{Remark}

\usepackage{dsfont,mathrsfs,euscript,nicefrac,multirow,arydshln}

\newcommand{\R}{\mathds R}

\newcommand{\skalarProd}[2]{\big\langle#1,#2\big\rangle}
\newcommand{\pnorm}[1]{\| #1 \|_{L^p(\Omega)}}
\renewcommand{\d}{\ \mathrm d}

\usepackage{tikz}

\begin{document}
\maketitle

\begin{tikzpicture}[remember picture, overlay]
 \node [xshift=-1cm,yshift=15cm,rotate=-90] at (current page.south east)
 {Electronic Journal of Differential Equations (2021), \href{https://ejde.math.txstate.edu/}{https://ejde.math.txstate.edu/}
 };
\end{tikzpicture}

\textit{AMS 2010 MSC:}{35B40, 35L70}

\begin{abstract}
In this note, we study the asymptotic behavior, as $t$ tends to infinity,  of the solution $u$ to the evolutionary  damped $p$-Laplace equation
\begin{equation*}
  u_{tt}+a\, u_t =\Delta_p u
\end{equation*}
 with  Dirichlet boundary values.  Let  $u^*$ denote the stationary solution with same boundary values, then we prove  the $W^{1,p}$-norm of $u(t) - u^{*}$ decays for large $t$ like $t^{-\frac{1}{(p-1)p}}$,  in the
degenerate case   $ p > 2$.
\end{abstract}

\numberwithin{equation}{section}

\allowdisplaybreaks

\section{Introduction and problem setting}
Let $\Omega\subset\R^n$ be a bounded Lipschitz domain, $p\geq2$ and $g\in W^{1,p}(\Omega)$. Consider the  minimization of  the following  functional
\begin{equation}\label{eq:BL:min}
\EuScript{E}(u) \coloneqq\frac1p\int_\Omega|\nabla u|^p \d x,
\end{equation}
over  the class $\EuScript{C}\coloneqq\{u:  u-g\in W^{1,p}_0(\Omega)\}$. The minimizer  denoted by  $u^*(x)$ satisfies the following Euler-Lagrange-equation in the weak sense:
\begin{equation}\label{eq:BL:stationary}
 \begin{cases}
  -\Delta_p u^{*} &= 0 \quad \text{in $\Omega$,}\\
  \hfill u^{*} &= g \quad \text{on $\partial\Omega$.}
 \end{cases}
\end{equation}

The first order flow  of  $\EuScript{E}(v)$,  i.e.  $v_t +\partial_{v} \EuScript{E}(v) = 0,$  can be considered as a classical steepest descent flow for solving the  minimization problem \eqref{eq:BL:min}.  In the degenerate case $p>2$ the authors of \cite{BL:Lindqvist} obtained  the sharp decay rate
\[
\sup_{x\in\Omega} |v(t,x)-u^{*}(x)| = O\left(t^{-\frac{1}{p-2}}\right) \quad \text{as $t\to\infty$.}
\]
Their proof  is based on the Moser iteration   applied  to the difference $v(t,x)-u^{*}(x)$, which
itself is not a solution, thus bounding the $ L^\infty$-norm in terms of the $ L^p$-norm.

It is well known, that an improvement in the convergence rate may be gained by considering the corresponding second order damped problem, cf. \cite{BL:Frankel,BL:Polyak,BL:BSGZ1} and references therein.  Moreover, second order damped problems naturally appear in modeling mechanical systems. For instance, the motion of a material point with positive mass sliding on a profile defined by a function $\Phi$ under the action of the gravity force, the reaction force, and the friction force can asymptotically be approximated by the following second order dynamical system
\begin{equation}
\ddot{x}(t) + \lambda\dot{x}(t)+\nabla \Phi(x(t))=0                                                                                                                                                                                                                                                                                                                                                                                                                                                                                                                                                                                           \end{equation}
called \emph{heavy ball with friction system (HBF)}, cf. \cite{BL:AGR}.  We refer to \cite{BL:GOOZ} and \cite{BL:ENEO} to see numerical algorithms based on the HBF system for  solving some special problems, e.g. large systems of linear equations, eigenvalue problems, nonlinear Schr\"odinger problems, inverse source problems, and  ill-posed problems. In \cite{BL:ENEO} the authors have shown advantages and superior convergence properties of such a dynamical functional particle method compared to  a first order dynamical system, and also to several other iterative methods. So, it's hardly surprising that second order dynamical equations play an important role in acceleration for convergence to steady state solutions. In fact, the power of the use of the damped $p$-Laplace equation in image denoising was investigated in \cite{BL:BSGZ1}. However, an analysis as in \cite{BL:Lindqvist} of the asymptotic behavior, as  $t\to\infty$, of the solutions to a damped $p$-Laplace equation was not done so far.

Our  purpose here is to obtain the decay rate for large time of $u-u^*$ where $u$ denotes a solution to the evolutionary damped $p$-Laplace equation\footnote{The question of existence of solutions will be the subject of a forthcoming note.}, namely:

 \newcommand{\negphantom}[1]{\settowidth{\dimen0}{$#1$}\hspace*{-\dimen0}}

\begin{equation}\label{eq:BL:telegraph}
 \begin{cases}
  u_{tt}+a\, u_t &=\Delta_p u\hphantom{u_{0}(x)}\negphantom{\Delta_p u} \quad \text{in $(0,\infty)\times\Omega$,}\\
  \hfill u(0,x)&=u_{0}(x) \quad\text{in $\{0\}\times\Omega$,}\\
  \hfill u_t(0,x)&=0\hphantom{u_{0}(x)}\negphantom{0} \quad \forall x\in\Omega,\\
  \hfill u(t,x)&=g(x)\hphantom{u_{0}(x)}\negphantom{g(x)} \quad \text{on $[0,\infty)\times\partial\Omega$,}
 \end{cases}
\end{equation}
wherein $a>0$ is constant and $u_0\in W^{1,p}(\Omega)$, such that $u_0-g\in W^{1,p}_0(\Omega)$.

It is clear, that the solution of the damped equation \eqref{eq:BL:telegraph} behaves for large time like the stationary solution of \eqref{eq:BL:stationary}. Moreover, we show  the following rate of decay for the $W^{1,p}$-norm of their difference:
\begin{theorem}\label{thm:BL:main}
 Let $p\geq2$, $u^*$ denote a solution to \eqref{eq:BL:stationary} and $u$ a solution to \eqref{eq:BL:telegraph}. For large time we have
 \begin{equation*}
  \|u-u^*\|_{W^{1,p}(\Omega)} \leq C \cdot t^{-\frac{1}{(p-1)p}},
 \end{equation*}
 with a constant $C=C(p,\Omega,u_0,a)>0$.
\end{theorem}

Our proof is based on a careful analysis of the following error term:
\begin{equation}\label{eq:BL:errordef}
 \mathrm{e}(t)\coloneqq\int_\Omega \frac{a^2}{2}w^2+a\,w\,w_t + w_t^2+2\cdot\left(\frac1p|\nabla u|^p-\frac1p|\nabla u^*|^p\right)\d x
\end{equation}
where we have set $w=u-u^*$. Note that our error term is chosen in such a way that it is compatible to our problem and we can estimate the error in terms of its derivative. Moreover, the fact
\[
  \frac{\d}{\d t}\int_\Omega\frac1p|\nabla u|^p \d x = -  \int_\Omega   u_t\,  \Delta_p u \d x,
\]
cf. page \pageref{eq:BL:integrated}, justifies the appearance of the last term in the error.  It is worth mentioning that with our argumentation scheme we can improve the decay rate in the linear case $p=2$ and obtain the classical result from \cite{BL:HZ}, cf. the discussion in section \ref{sec:BL:improved}.

\section{Basic  results}

Let us briefly introduce the notations used throughout this work.
The Euclidean norm  in  $\mathbb{R}^n$ is denoted by $|\cdot|$, a generic positive constant is  represented  by capital or small letter $c$ possibly varying from line to line, and we often write $u(t)(x)$ for $u(t,x)$.

 Given a real Banach space $X$, the (Banach) space $L^{p}(0,T; X)$  consists of
 all measurable
 functions $u:[0,T]\to X$ such that
\[
\|u\|_{L^p(0,T;X)}=\left(\int_{0}^{T} \|u(t)\|_{X}^{p}\, \d t\right)^{\frac{1}{p}}
<\infty\, , \qquad  1\leq p<\infty\, ,
\]
 $L^{\infty}(0,T; X)$  is the space of all measurable
 $u:[0,T]\to X$ such that
\[
\|u\|_{L^{\infty}(0,T; X)}=\underset{t\in [0,T]}{\text{ess sup}} \|u(t)\|_{X}.
<\infty.
\]
The Banach space $W^{1,p}(0,T; X)$, for $1\leq p  \leq \infty$, consists of all $u\in L^{p}(0,T; X)$ such that $\partial_{t}u$ exists in the weak sense and belongs to $L^{p}(0,T; X)$.

Recall that for $u\in W^{1,p}(0,T; X)$ we have $u\in C^0([0, T]; X)$ and
$$\max_{0\le t\le T}\|u(t)\|_{X}\leq c(T)\cdot \|u\|_{W^{1,p}(0,T; X)}.$$

For further reading and elaborated clarifications on spaces involving time,  we refer the reader to \cite[Sec. 5.9.2]{BL:Evans}.

Throughout this work, we make use of the following inequalities:
\begin{itemize}
\item let $p\geq 2$. For all $a,b\in \R^n$ we have
\begin{equation}\label{eq:BL:VectIne}
 2^{2-p}|a-b|^p\leq \skalarProd{|a|^{p-2}a-|b|^{p-2}b}{a-b}, \tag{A1}
\end{equation}

\item
for $p\ge 2$ and with an adequate constant $c(p)\in(0,1]$:
\begin{equation}\label{eq:BL:VectIne2}
 |b|^p\geq |a|^p+p\skalarProd{|a|^{p-2}a}{b-a}+c(p)|b-a|^p, \tag{A2}
\end{equation}

\item
 furthermore, for ~ $\pnorm{f}\leq M$ ~ and ~ $\pnorm{g}\leq M$ ~ the estimate
\begin{equation}\label{eq:BL:Rudin}
 \int_\Omega \left | \vphantom{x^{x^{x^x}}}|f|^p -|g|^p\right| \d x \leq c(p,\Omega) M^{p-1} \pnorm{f-g} \tag{A3}
\end{equation}
holds, cf. \cite[p.~75]{BL:Rudin}.
\end{itemize}

Firstly, let us define the concept of weak  solutions to the evolutionary damped $p$-Laplace equation:
\begin{definition}
We say that  $u  \in W^{1,p}_{\textrm{loc}}( 0,\infty ; W^{1,p}(\Omega)) $  is a solution to
 \begin{equation*}
  u_{tt}+a\, u_t =\Delta_p u
\end{equation*}
 if
\[
 \int_{0}^{\infty} \int_\Omega  - u_{t}\, \phi_{t}-  a\,  u\, \phi_{t} +  |\nabla u|^{p-2}\skalarProd{\nabla u}{\nabla \phi}\d x \d t =0,
\]
for each  $\phi \in C^{\infty}_{0}( (0, \infty) \times \Omega).$
\end{definition}

In the following, let us denote by $u^*$ a solution to \eqref{eq:BL:stationary} and by $u$ a solution to \eqref{eq:BL:telegraph}. Moreover, we set
\[
 E(t)\coloneqq\int_\Omega\frac12\,u_t^2(t,x) + \frac1p\,|\nabla u(t,x)|^p  \d x.
\]
\begin{corollary}
 $E(\cdot)$ is non-increasing, or rather in the weak sense we have
 \begin{equation}\label{eq:BL:Lemma1}
  E'(t)= -a \int_\Omega u_t^2 \d x.
 \end{equation}
  \end{corollary}
\begin{proof}
 A multiplication of
$$
u_{tt}+a\, u_t =\Delta_p u
$$
with $u_t$ followed by an integration over $\Omega$ gives

\begin{equation} \label{eq:BL:integrated}
\int_{\Omega}   u_{tt} \, u_t \d x -\int_{\Omega} (\Delta_p u )    \, u_t \d x = -a   \int_{\Omega}   {u_t^2}  \d x .
\end{equation}
Moreover, an integration by parts  yields
\[
 -  \int_\Omega   u_t\,  \Delta_p u \d x = \frac{\d}{\d t}\int_\Omega\frac1p|\nabla u|^p \d x,
\]
note that there is no time dependence of $u_t$ on the boundary.  In view of
\[
\int_{\Omega}   u_{tt} \, u_t \d x= \frac{1}{2} \frac{\d}{\d t} \int_{\Omega}  u_{t}{^2} \d x,
\]
we combine the last two equalities with \eqref{eq:BL:integrated} to achieve the desired relation \eqref{eq:BL:Lemma1}:
\begin{equation*}
E'(t)=\frac{\d}{\d t} \int_{\Omega} \frac12{u_t^2} \d x+  \frac{\d}{\d t}\int_\Omega\frac1p|\nabla u|^p \d x=-a \int_\Omega u_t^2 \d x.\qedhere
 \end{equation*}
\end{proof}

\begin{remark}
 The above considerations were formal and can all be made rigorous, cf. e.g. \cite[p. 156ff]{BL:Wu}.
\end{remark}

In view of \eqref{eq:BL:Lemma1}, we show that the gradient of $u$ (with respect to space) is bounded by the initial data and that $u_t$ tends to zero for big times:

\begin{corollary}\label{cor:BL:ut} Let $u$ be a solution to \eqref{eq:BL:telegraph}. Then:
 \begin{enumerate}[a)]
\item We have \quad $\|u_t(T)\|_{L^2(\Omega)}\xrightarrow{T\to\infty}0$.
\item For all $T\geq 0$ it  holds \quad $ \|\nabla u(T) \|_{L^p(\Omega)} \leq \pnorm{\nabla u_0}.$
 \end{enumerate}
 \end{corollary}
\begin{proof}

Integrating  \eqref{eq:BL:Lemma1} over $(0,T)$ we gain

\begin{align}\label{sun3}
  \int_{\Omega} \frac12u_t^2(T,x) \d x+   \int_\Omega\frac1p|\nabla u(T, x)|^p \d x  + a \int_{0}^{T}\int_\Omega & u_t^2(\tau,x)  \d x \d \tau \notag\\
   &\le  \int_\Omega\frac1p|\nabla u_{0}(x)|^p\d x.
 \end{align}
Note that the right hand  side of inequality \eqref{sun3} is independent of $T$, hence, the statement follows with $T\to\infty$.
\end{proof}

\begin{remark}\label{rem:BL:reg}
Taking the essential supremum with respect to time on both sides of \eqref{sun3} shows
\[
u_t \in L^{\infty}( 0, \infty; L^{2}(\Omega)), \quad \text{and}\quad u \in L^{\infty}( 0, \infty; W^{1,p}(\Omega)). \qedhere
\]
\end{remark}

Recall that $u^*$ minimizes  $ \EuScript{E}(\cdot) $. Hence, Corollary \ref{cor:BL:ut} ensures the boundedness of the gradients of both $u$ and $u^*$,  more precisely
\begin{equation}\label{eq:BL:gradientbound}
 \pnorm{\nabla u^*}\le\pnorm{\nabla u}\le M
\end{equation}
where we have set $M\coloneqq\pnorm{\nabla u_0}$.

Next, let us focus on the behavior of the energies. Since for large time the dependence of $u$ on time shrinks, cf. Cor. \ref{cor:BL:ut}, the convergence of energies should follow from the uniqueness of $p$-harmonic functions, and indeed, we have the following

\begin{lemma}\label{lem:BL:convergenceGrad}
  Let $u^{*}$  and    $u$ be solutions  of  \eqref{eq:BL:stationary} and \eqref{eq:BL:telegraph}, respectively, then
  \[
  \EuScript{E}(u)\xrightarrow{t\to \infty}\EuScript{E}(u^*).
  \]
\end{lemma}
\begin{proof}
Since  $u^*$ is the unique minimizer  of  $\EuScript{E}(\cdot)$, it suffices to show that
 \begin{equation} \label{eq:BL:AlvarezGoal}
\limsup_{t\to\infty} \frac1p\int_\Omega|\nabla u|^p\d x \leq \frac1p\int_\Omega|\nabla v|^p\d x
 \end{equation}
for all $v$ such that $v-g\in W^{1,p}_0(\Omega)$. For that purpose we will basically follow the proof of Theorem 2.1 from \cite{BL:Alvarez}:

Let $v\in W^{1,p}(\Omega)$ with $v-g\in W^{1,p}_0(\Omega)$ be given. Consider the following auxiliary function
\[
\varphi(t)\coloneqq \frac12\int_\Omega \left(u(t,x)-v(x)\right)^2\d x.
\]
Then $\varphi\in W^{2,1}(0,\infty)$, cf. Remark \ref{rem:BL:reg}, and, as $u$ fulfills \eqref{eq:BL:telegraph}, we have
\begin{align*}
 \varphi''(t)+a\,\varphi'(t)&=\int_\Omega (u-v)\,\Delta_p u +u_t^2\d x= -\int_\Omega |\nabla u|^{p-2} \skalarProd{\nabla u}{\nabla u - \nabla v}+u_t^2\d x \\
 &\overset{\eqref{eq:BL:VectIne2}}{\leq}\int_\Omega \frac1p|\nabla v|^p-\frac1p|\nabla u|^p+u_t^2\d x\leq \int_\Omega \frac1p|\nabla v|^p+\frac32u_t^2\d x -E(T)
\end{align*}
for all $t\in[0,T]$, where we have used that $E(\cdot)$ is non-increasing. A multiplication of both sides with $e^{at}$, followed by an integration yields
\begin{align*}
\varphi'(t)\le  e^{-at} \varphi'(0) &+ \frac{1}{a}(1- e^{-at}) \left(\int_\Omega \frac1p|\nabla v|^p\d x- E(T)\right)\\
&+ \frac{3}{2}  \int_0^t\int_\Omega e^{-a(t-\tau)}{u_t^2(\tau,x)}\d x \d\tau.
\end{align*}
Integrating once more and using the fact that $\displaystyle E(T)\ge \frac1p\int_\Omega|\nabla u|^p\d x$,  implies
\begin{align}\label{eq:BL:Alvarez}
 &\varphi(T)+\frac{1}{a^2}\left(aT-1 +e^{-aT}\right)\frac1p\int_\Omega|\nabla u|^p\d x\notag\\
 &\leq \frac{1}{a^2}\left(aT-1+e^{-aT}\right)\frac1p\int_\Omega|\nabla v|^p\d x\
  + \varphi(0)+\frac1a(1-e^{-aT})\,\varphi'(0)+h(T)
 \end{align}
where we have set
\begin{align*}
 h(T)&\coloneqq\frac{3}{2}\int_0^T\int_0^t\int_\Omega e^{-a(t-\tau)}{u_t^2(\tau,x)}\d x \d \tau \d t\\
 &=\frac{3}{2a}\int_0^T\int_\Omega {u_t^2(\tau,x)}(1-e^{-a(T-\tau)})\d x\d \tau.
\end{align*}
Due to Remark \ref{rem:BL:reg} the term $h(T)$ is bounded. Hence, dividing \eqref{eq:BL:Alvarez} by  the expression $\displaystyle \frac{1}{a^2}\left(aT-1 +e^{-aT}\right)$ and letting $T\to\infty$ gives the desired estimate \eqref{eq:BL:AlvarezGoal}.
\end{proof}

On account of the convergence of the energies, we get the $W^{1,p}$ convergence of $u$ to $u^*$:

\begin{corollary}\label{cor:BL:u}
Let $u$ and $u^*$ be as before, then we have $$\|{u-u^*}\|_{W^{1,p}(\Omega)}\xrightarrow{t\to\infty}0.$$
 \end{corollary}

 \begin{proof}
 By Poincar\'e's inequality we have
 \begin{equation}
 \int_\Omega |u-u^*|^p \d x \leq \tilde{c}(p,\Omega) \int_\Omega |\nabla u  - \nabla u^*|^p \d x.
   \end{equation}
For the $p$-harmonic function $u^*$ it holds
\begin{equation*}
 \int_\Omega |\nabla u^*|^{p-2}\skalarProd{\nabla u^*}{\nabla u - \nabla u^*} = 0,
\end{equation*}
 so that, by \eqref{eq:BL:VectIne2} we obtain
\begin{equation}\label{eq:BL:estimate}
c(p) \int_\Omega |\nabla u - \nabla u^*|^p \d x\leq  \int_\Omega|\nabla u|^p -|\nabla u^*|^p \d x=p\cdot(\EuScript{E}(u)-\EuScript{E}(u^*)).
\end{equation}
Thus, combining the above estimates we arrive at
\begin{equation*}
 \|{u-u^*}\|_{W^{1,p}(\Omega)}\le c(p,\Omega)\cdot |\EuScript{E}(u)-\EuScript{E}(u^*)|\xrightarrow{t\to\infty}0
\end{equation*}
by Lemma \ref{lem:BL:convergenceGrad}.
\end{proof}

\section{Proof of the decay rate}
We are now prepared to prove our main result:
\begin{proof}[Proof of Theorem \ref{thm:BL:main}]
A multiplication of
$$
u_{tt}+a\, u_t =\Delta_p u - \Delta_p u^*
$$
with $w=w(t,x)\coloneqq u(t,x)-u^*(x)$, and integrating by parts (note that $\left.w\right|_{\partial\Omega}=0$) yields
\begin{align*}
 \int_\Omega w_{tt}\,w+a\,w_t\,w \d x &= - \int_\Omega\skalarProd{|\nabla u|^{p-2}\nabla u -|\nabla u^*|^{p-2}\nabla u^* }{\nabla u - \nabla u^*}\d x\\
 & \overset{\eqref{eq:BL:VectIne}}{\leq} -2^{2-p}\int_\Omega |\nabla w|^p \d x.
\end{align*}
Hence, multiplying both sides of the last inequality with $a>0$ and adding ~ $\int_\Omega a\,w_t^2$ ~ we end up with
\begin{equation}\label{eq:BL:8}
 \frac{\d}{\d t}\int_\Omega \frac{a^2}{2}w^2+a\,w\,w_t \d x \leq \int_\Omega a\,w_t^2-2^{2-p}a|\nabla w|^p\d x.
\end{equation}
Recall the definition of our error term
\begin{equation}\tag{\ref{eq:BL:errordef}}
 \mathrm{e}(t)\coloneqq \int_\Omega \frac{a^2}{2}w^2+a\,w\,w_t + w_t^2+2\cdot\left(\frac1p|\nabla u|^p-\frac1p|\nabla u^*|^p\right)\d x.
\end{equation}
So, $\mathrm{e}\in W^{1,1}(0,\infty)$ and due to the minimizing properties of ~$u^*=u^*(x)$,~ we have that ~ $\mathrm{e}(t)\geq0$ for all $t>0$. ~ Since $w_t=u_t$ we obtain
\begin{align}
 \mathrm{e}'(t) &\overset{\eqref{eq:BL:8}}{\leq} \int_\Omega a\,w_t^2-2^{2-p}a|\nabla w|^p\d x + \frac{\d}{\d t}\int_\Omega u_t^2 + 2\cdot\frac1p|\nabla u|^p  \d x \notag\\
 &\overset{\eqref{eq:BL:Lemma1}}{=} -a \int_\Omega w_t^2+2^{2-p}|\nabla w|^p \d x\label{eq:BL:abl} \leq 0.
\end{align}
Moreover, again with $w_t=u_t$ we have
\begin{align}\label{eq:hilf1e}
 \frac{\mathrm{e}'(t)}{a} &= \int_\Omega a\,w\,w_t+w\,w_{tt}+w_t^2\d x+\frac{2}{a}\frac{\d}{\d t}\int_\Omega \frac12 w_t^2+\frac1p|\nabla u|^p\d x\notag\\
 &\overset{\mathclap{\eqref{eq:BL:Lemma1}}}{=} \ \int_\Omega a\,w\,w_t+w\,w_{tt}-w_t^2\d x = \int_\Omega w(\Delta_p u - \Delta_p u^*) -w_t^2\d x \notag \\
 &= \int_\Omega w \Delta_p u -w_t^2\d x,
 \end{align}
 since $\Delta_p u^*=0$ in $\Omega$. Using integration by parts (note that $\left.w\right|_{\partial\Omega}=0$) we obtain 
\begin{align}\label{eq:hilf2e}
\left| \int_\Omega w \Delta_p u \d x \right | &= \left| \int_\Omega |\nabla u |^{p-2}\skalarProd{\nabla u}{\nabla w}\d x \right| \le  \int_\Omega |\nabla u|^{p-1} |\nabla w|\d x \notag \\
& \le \pnorm{\nabla u}^{p-1}\pnorm{\nabla w}
\end{align}
by Hölder's inequality. Using the boundedness of the gradient \eqref{eq:BL:gradientbound} we conclude: 
\begin{align}
 \left|\frac{\mathrm{e}'(t)}{a}\right|& \overset{\eqref{eq:hilf1e}}=\left|\int_\Omega w \Delta_p u -w_t^2\d x\right| \overset{\eqref{eq:hilf2e}}{\le} \pnorm{\nabla u}^{p-1}\pnorm{\nabla w} +\|w_t\|_{L^2(\Omega)}^2\notag\\
 &\overset{\eqref{eq:BL:gradientbound}}{\le}M^{p-1}\pnorm{\nabla w}+\|w_t\|_{L^2(\Omega)}^2
 \xrightarrow{t\to\infty}0,\label{eq:BL:derivative}
\end{align}
by Corollary \ref{cor:BL:u} and Corollary \ref{cor:BL:ut}, respectively.

Our next goal is to estimate the error in terms of its derivative. In regard with \eqref{eq:BL:Rudin} we arrive at
\begin{equation*}
 \mathrm{e}(t) \leq \int_\Omega\left(\frac{a^2}{2}+a\right)w^2 + \left(\frac{a}{4}+1\right)w_t^2 \d x + c(p,\Omega,u_0)\cdot\pnorm{\nabla w}\,.
\end{equation*}
 Using Lebesgue embedding and Poincar\'e's inequality for the first term we get
\begin{equation*}
 \mathrm{e}(t) \leq c_1(p,\Omega,a)\cdot\pnorm{\nabla w}^2 + \left(\frac{a}{4}+1\right)\int_\Omega w_t^2 \d x + c(p,\Omega,u_0)\cdot\pnorm{\nabla w}\,.
\end{equation*}
Furthermore, in \eqref{eq:BL:abl} we already aimed
\[
 \int_\Omega w_t^2+2^{2-p}|\nabla w|^p \d x \leq -\frac{\mathrm{e}'(t)}{a}.
\]
All in all, we get
\begin{equation*}
 \mathrm{e}(t) \leq c_2(p,\Omega,a)\cdot \left(-\frac{\mathrm{e}'(t)}{a}\right)^{\frac{2}{p}} + \left(\frac{a}{4}+1\right)\cdot \left(-\frac{\mathrm{e}'(t)}{a}\right) + c_3(p,\Omega,u_0)\cdot\left(-\frac{\mathrm{e}'(t)}{a}\right)^{\frac{1}{p}}.
\end{equation*}

Since $$ -\frac{\mathrm{e}'(t)}{a} \xrightarrow{t\to\infty} 0,$$ cf. \eqref{eq:BL:derivative}, the error term $\mathrm{e}(t)\geq0$ satisfies for large time a differential inequality of type
\begin{align}
 \mathrm{e}(t) &\leq c_4(p,\Omega,u_0,a)\cdot (-\mathrm{e}'(t))^{\frac1p},\notag\\
 \shortintertext{and we may rewrite this}
 \mathrm{e}'(t) &\leq - c_5(p,\Omega,u_0,a)\cdot \mathrm{e}(t)^p,\notag\\
\shortintertext{respectively, so by Lemma 1.6 from \cite{BL:HZ} we gain}
 \mathrm{e}(t)&\leq c_6(p,\Omega,u_0,a)\cdot t^{-\frac{1}{p-1}}.
\label{eq:BL:decayRate1}
\end{align}

By \eqref{eq:BL:decayRate1}, \eqref{eq:BL:estimate} and the Poincar\'e inequality we finally arrive at
\begin{equation*}
 \|u-u^*\|_{W^{1,p}(\Omega)}^p \leq c_7(p,\Omega,u_0,a) \cdot t^{-\frac{1}{p-1}}\,.\qedhere
\end{equation*}
\end{proof}

\subsection{Enhancement  of the decay rate for \boldmath $p=2$}\label{sec:BL:improved}
A crucial ingredient in our proof of the decay rate was inequality \eqref{eq:BL:Rudin} which we applied to estimate the difference of the energies. In fact, for $p=2$ this relation can be improved to the \textit{equality}
\begin{equation*}
 \int_\Omega |\nabla u|^2 -|\nabla u^*|^2 \d x = \int_\Omega |\nabla u - \nabla u^*|^2 \d x
\end{equation*}
where we used the harmonicity of $u^*$. Hence, we obtain for the error term
\begin{align}
 \mathrm{e}(t) & \leq \int_\Omega \left(\frac{a^2}{2}+a\varepsilon\right)w^2 + \left(\frac{a}{4\varepsilon}+1\right)w_t^2+|\nabla w|^2\d x\notag\\
 &\le \int_\Omega\left(\left(\frac{a^2}{2}+a\,\varepsilon\right)\tilde{c}(\Omega)+1\right)|\nabla w|^2 +\left(\frac{a}{4\varepsilon}+1\right)w_t^2\d x\notag\\
 & = c(a,\Omega) \int_\Omega w_t^2+|\nabla w|^2 \d x\le c(a,\Omega)\left(-\frac{\mathrm{e}'(t)}{a}\right)\label{eq:BL:p2}
\end{align}
where in the intermediate steps we used the Poincar\'e inequality, and $\varepsilon>0$ was choosen in such a way that the prefactors coincided. Relation \eqref{eq:BL:p2} may be rewriten to
\[
\mathrm{e}'(t)\le -\frac{a}{c(a,\Omega)}\,\mathrm{e}(t) \quad \text{for all $t>0$},
\]
so, by Gronwall's inequality, the error term fulfills
\[
   \mathrm{e}(t) \le c\cdot \exp\left(-\frac{a}{c(a,\Omega)}\ t\right)
\]
and for the decay rate we arrive at
\[
 \|u-u^*\|^2_{W^{1,2}(\Omega)}\le C \cdot \exp\left(-\frac{a}{c(a,\Omega)}\ t\right) \quad \text{for all $t>0$},
\]
a well known result, cf. e.g. Theorem 2.1 a) in \cite{BL:HZ}.

\subsection*{Acknowledgments}
The authors are grateful to   the  Hausdorff Research Institute for Mathematics (Bonn) for support
and hospitality during the trimester program \emph{Evolution of Interfaces}, where work on
this article was undertaken.  Moreover, the authors  would
like to thank John Andersonn for helpful suggestions and discussion.

\end{document}